\documentclass[11 pt]{article}
\usepackage{amsmath,amsfonts,amssymb,amsthm}
\usepackage{cases}
\usepackage{color}
\usepackage{xcolor}
\usepackage{colortbl}
\usepackage{graphicx}
\usepackage{hyperref}

\usepackage{cite}
\newcommand{\Z}{\mathbb{Z}}

\usepackage{todonotes}

\allowdisplaybreaks
\numberwithin{equation}{section}

\newtheorem{theorem}{Theorem}
\newtheorem{lemma}{Lemma}

\newtheorem{proposition}{Proposition}

\numberwithin{theorem}{section}
\numberwithin{lemma}{section}
\numberwithin{corollary}{section}
\numberwithin{proposition}{section}
\numberwithin{definition}{section}
\numberwithin{remark}{section}
\numberwithin{example}{section}
\numberwithin{equation}{section}

\date{}

\begin{document}

\title{\bf Balance and pattern distribution of sequences derived from pseudorandom subsets of $\Z_q$ }
\author{\begin{tabular}[t]{c@{\extracolsep{0em}}c}
{\large Huaning Liu$^{1}$ and Arne Winterhof$^{2}$  }  \\ \\
  {\normalsize {$^{1}$Research Center for Number Theory and Its Applications }} \\
  {\normalsize {School of Mathematics, Northwest University }} \\
   {\normalsize {Xi'an 710127, China }} \\
  {\normalsize {E-mail: hnliu@nwu.edu.cn }} \\
  {\normalsize {$^{2}$Johann Radon Institute for Computational and Applied Mathematics }} \\
  {\normalsize {Austrian Academy of Sciences }} \\
   {\normalsize {Altenbergerstr.\ 69, 4040 Linz, Austria }} \\
  {\normalsize {E-mail: arne.winterhof@oeaw.ac.at }}
  \end{tabular}}

\maketitle

\begin{center}
    In memory of Reinhard Winkler\\~
\end{center}

\begin{abstract} Let $q$ be a positive integer and $\mathcal{S}=\left\{x_0,x_1,\ldots,x_{T-1}\right\}\subseteq\mathbb{Z}_q=\{0,1,\ldots,q-1\}$
with
$$0\leq x_0<x_1<\ldots< x_{T-1}\leq q-1.$$
We derive from $\mathcal{S}$ three (finite) sequences.
\begin{enumerate}
\item For an integer $M\geq 2$ let $(s_n)$ be the $M$-ary sequence defined by
\begin{eqnarray*}
s_n\equiv x_{n+1}-x_n \bmod M, \qquad n=0,1,\ldots, T-2.
\end{eqnarray*}
\item For an integer $m\geq 2$ let $(t_n)$ be the binary sequence defined by
\begin{eqnarray*}
t_n=\left\{\begin{array}{ll}
1, & \hbox{if \ }  1\leq x_{n+1}-x_n\leq m-1, \\
0, & \hbox{otherwise},
\end{array}\right. \qquad n=0,1,\ldots, T-2.
\end{eqnarray*}
\item Let $(u_n)$ be the characteristic sequence of $\mathcal{S}$,
\begin{eqnarray*}
u_n=\left\{\begin{array}{ll}
1, & \hbox{if \ }  n\in \mathcal{S}, \\
0, & \hbox{otherwise},
\end{array}\right. \qquad n=0,1,\ldots, q-1.
\end{eqnarray*}
\end{enumerate}
We study the balance and pattern distribution of the sequences
$(s_n)$, $(t_n)$ and $(u_n)$. For sets~$\mathcal{S}$ with desirable pseudorandom properties, more precisely, sets with low correlation measures, we show the following:
\begin{enumerate}
\item  The sequence $(s_n)$ is (asymptotically) balanced and has uniform pattern distribution if $T$
is of smaller order of magnitude than $q$.

\item The sequence $(t_n)$ is balanced and has uniform pattern distribution if
$T$ is approximately $\left(1-\frac{1}{2^{1/(m-1)}}\right)q$.
\item The sequence $(u_n)$ is
balanced and has uniform pattern distribution if $T$ is
approximately $\frac{q}{2}$.
\end{enumerate}
These results are motivated by earlier results for the sets of quadratic residues and primitive roots modulo a prime.
We unify these results and derive many further
(asymptotically) balanced sequences with uniform pattern distribution
from pseudorandom subsets.  \\


\end{abstract}

\section{Introduction}

The balance and pattern distribution of several sequences derived from pseudorandom subsets of the finite field $\Z_p$ of prime order $p$ such as the set of quadratic residues and the set of primitive roots modulo $p$ have already been studied in the literature, see in particular \cite{Winterhof2021,Winterhof2022}. In this paper we generalize the approach
of \cite{Winterhof2021,Winterhof2022} for the sets of quadratic residues and primitive roots modulo $p$ to any pseudorandom subset. This unifies previous results and provides many new results for free.

\subsection{Fundamental examples}

More precisely, let $p>2$ be a prime, $q_0,q_1,\ldots,q_{\frac{p-3}{2}}$ be the quadratic residues modulo
$p$ in increasing order and let $g_0,g_1,\ldots,g_{\varphi(p-1)-1}$ be the primitive roots
modulo $p$ in increasing order, where $\varphi$ is  Euler's totient function.
Define the (finite) binary sequences $(s_n')$, $(s_n'')$, $(t_n')$, $(t_n'')$,
$(u_n')$ and $(u_n'')$ by
\begin{eqnarray*}
s_n'&\equiv& q_{n+1}-q_n\bmod 2, \qquad n=0,1,\ldots,\frac{p-5}{2},\\
s_n''&\equiv& g_{n+1}-g_{n} \bmod 2, \qquad n=0,1,\ldots,\varphi(p-1)-2, \\
t_n'&=&\left\{\begin{array}{ll}
1, & \hbox{if \ } q_{n+1}-q_n=1, \\
0, & \hbox{otherwise},
\end{array}\right. \qquad n=0,1,\ldots,\frac{p-5}{2}, \\
t_n''&=&\left\{\begin{array}{ll}
1, & \hbox{if \ } g_{n+1}-g_n=1, \\
0, & \hbox{otherwise},
\end{array}\right. \qquad n=0,1,\ldots,\varphi(p-1)-2,\\
u_n'&=&\left\{\begin{array}{ll}
1, & \hbox{if $n$ is a quadratic residue modulo $p$},\\
0, & \hbox{otherwise},
\end{array}\right. \quad n=0,1,\ldots, p-1,\\
u_n''&=&\left\{\begin{array}{ll}
1, & \hbox{if $n$ is a primitive root modulo $p$},\\
0, & \hbox{otherwise},
\end{array}\right. \quad n=0,1,\ldots,p-1.
\end{eqnarray*}
A (finite) binary sequence is balanced if the numbers of sequence elements equal to $0$ and $1$ differ by at most $1$.
(Note that the term balanced is used with a different meaning in combinatorics on words, see for example \cite[Definition~10.5.4]{Allouche2003}.) A sequence of length $N$ is asymptotically balanced if this number is $o(N)$.
Here
$$
f(n)=o(g(n))\quad \mbox{if}\quad
\lim_{n\rightarrow \infty} \frac{f(n)}{g(n)}=0.
$$
A sequence $(s_n)$ of length $N$ has (asymptotically) uniform pattern distribution (of length~$\ell$)
if each pattern in $\{0,1\}^\ell$ of length $\ell$ appears $N2^{-\ell}+o(N)$ times in the vector sequence
$$(s_n,s_{n+1},\ldots,s_{n+\ell-1}),\quad n=0,1,\ldots,N-\ell-1.$$
The extension of the terms balance and uniform pattern distribution to non-binary sequences is obvious.

Since there are $(p-1)/2$ quadratic residues and $\varphi(p-1)$ primitive roots modulo $p$,
$(u_n')$ is balanced and $(u_n'')$ is asymptotically balanced if
$\varphi(p-1)= \frac{p}{2}+o(p)$, which is true for example for Fermat primes, that is, $p$ is of the form $p=2^k+1$, and safe primes, that is, $(p-1)/2$ is also a (Sophie Germain) prime.
If $\varphi(p-1)$ is close to its maximum $(p-1)/2$, then almost all quadratic non-residues are primitive roots and the sequences $(u_n')$ and $(u_n'')$ are essentially dual sequences with essentially the same uniform pattern distribution. Hence, there is no need to study $(u_n'')$ in this case.
Ding \cite{Ding1998} proved that~$(u_n')$ has a uniform pattern distribution.

The second author and Xiao \cite{Winterhof2021} showed that the sequence
$(s_n'')$ is (asymptotically) balanced and has uniform pattern distribution provided that $\varphi(p-1)=o(p)$ and proved
in \cite{Winterhof2022} that the sequence $(t_n')$ is essentially balanced and has uniform pattern distribution but the
sequence~$(s_n')$ is quite unbalanced.
With the same methods one can show that~$(t_n'')$
is balanced and has desirable pattern distribution if $\varphi(p-1)= \frac{p}{2}+o(p)$.

The desirable features of the above sequences are induced by certain pseudorandomness properties of the sets
of quadratic residues and primitive roots modulo $p$. Our goal is to construct and analyze more such sequences derived from pseudorandom subsets of the residue class ring~$\Z_q$ modulo~$q$ which we identify with the integers between $0$ and $q-1$, where $q$ may be composite. The desired pseudorandom properties are measured in terms of the correlation measures defined below.

\subsection{The general case}

Pseudorandom subsets have been studied in a series of papers (see \cite{Chen2010,DartygeMS2009,DartygeS2007,DartygeS2007_1,DartygeS2009,DartygeSS2010,LiuQ2017,LiuS2014}),
in particular by
Dartyge and S\'{a}rk\"{o}zy (partly with other
coauthors). More precisely,
let $q$ be a positive integer, $\mathcal{R}\subseteq \mathbb{Z}_q=\{0,1,\ldots, q-1\}$ and define
\begin{equation}\label{en}
f_{\mathcal{R}}(n)=\left\{\begin{array}{rl} 1-\frac{|\mathcal{R}|}{q}, & \hbox{for
\ } n\in\mathcal{R},\\
-\frac{|\mathcal{R}|}{q}, & \hbox{for \ } n\not\in\mathcal{R}.
\end{array}\right.
\end{equation}
Dartyge and S\'{a}rk\"{o}zy \cite{DartygeS2009} introduced the
{\em correlation measure $C_k({\cal R},q\}$ of order
$k$} of a subset~$\mathcal{R}$ of $\Z_q$ as
$$
C_k(\mathcal{R},q)=\mathop{\max_{1\leq M\leq q}}_{0\leq d_1<d_2<\ldots<d_k\leq q-1}
\left|\sum_{n=0}^{M-1} f_{\mathcal{R}}(n+d_1)\ldots
f_{\mathcal{R}}(n+d_k)\right|.
$$
Obviously we have the trivial bound
$$
C_k(\mathcal{R},q)\le \min\{|\mathcal{R}|,q-|\mathcal{R}|\}.
$$

The subset $\mathcal{R}$ is considered a pseudorandom subset of $\Z_q$ if
$C_k(\mathcal{R},q)$ is ``small" (for
all $k=1,\ldots,K$ for some sufficiently large $K$), that is, $o\left(\min\{|\mathcal{R}|, q-|\mathcal{R}|\}\right)$.
In particular, very small $\mathcal{R}$
are not considered pseudorandom and in some of our sequence constructions below we may restrict ourselves to sufficiently large~$\mathcal{R}$.
(Note that the expected size of a random $\mathcal{R}$ is $\frac{q}{2}$,
its variance is $\frac{q}{4}$ and
by Chebyshev's inequality the probability that a random subset ${\cal R}$ satisfies $||{\cal R}|-q/2|\ge q^{1/2}\log q$ is $o(1)$ assuming that each element of $\Z_q$ belongs to ${\cal R}$ with probability $1/2$, that is, a binomial distribution.)
However, we also provide a promising sequence construction from rather sparse subsets.
For $\mathcal{R}\subseteq \mathbb{Z}_q$ we write
\begin{eqnarray*}
C(\mathcal{R},q,s)=\max_{1\leq k\leq s}C_k(\mathcal{R},q).
\end{eqnarray*}
Dartyge and S\'ark\"ozy also introduced the correlation measure of order $k$ for subsets of $\{1,\ldots,q\}$ in \cite{DartygeS2007}, that is, without reduction modulo $q$. Some bounds in the literature refer to this slightly different correlation measure. However, both measures are of the same order of magnitude and differ only by a multiplicative constant between $1$ and $2$.

Throughout this paper we denote
\begin{eqnarray}\label{Sdef}
\mathcal{S}=\left\{x_0,x_1,\ldots,x_{T-1}\right\}\subseteq\mathbb{Z}_q \quad\hbox{with}\quad
0\leq x_0<x_1<\ldots< x_{T-1}\leq q-1.
\end{eqnarray}
For given integers $M\geq 2$ and $m\geq 2$, let $(s_n)$ be
the $M$-ary sequence defined by
\begin{eqnarray}\label{snMdef}
s_n\equiv x_{n+1}-x_n \bmod M, \qquad n=0,1,\ldots, T-2,
\end{eqnarray}
and let $(t_n)$ be the binary sequence defined by
\begin{eqnarray}\label{tndef}
t_n=\left\{\begin{array}{ll}
1, & \hbox{if \ }  1\leq x_{n+1}-x_n\leq m-1, \\
0, & \hbox{otherwise},
\end{array}\right. \qquad n=0,1,\ldots, T-2.
\end{eqnarray}
We also consider the characteristic sequence $(u_n)$ of $\mathcal{S}$,
\begin{eqnarray}\label{undef}
u_n=\left\{\begin{array}{ll}
1, & \hbox{if \ }  n\in \mathcal{S}, \\
0, & \hbox{otherwise},
\end{array}\right. \qquad n=0,1,\ldots, q-1.
\end{eqnarray}

Clearly the sequence $(u_n)$ is balanced if and only if $T=\frac{q}{2}+o(q)$.
We study the balance and pattern distribution of the sequences
$(s_n)$ and $(t_n)$, and the pattern distribution of the sequence~$(u_n)$
in Section~\ref{sec:balance} and give large families of examples in Section~\ref{sec:examples}.

We write $f(n)=O(g(n))$ or $f(n)\ll g(n)$ if $|f(n)|\le c g(n)$ for some
absolute constant $c>0$.

\section{Balance and pattern distribution }\label{sec:balance}

For an $M$-ary sequence $(e_n)$ of length $T-1$ and $u\in \{ 1,\ldots,M\}$
let $N^{(u)}(e_n)$ denote the number of $n=0,1,\ldots,T-2$ with $e_n\equiv u\bmod M$.
If
$$
N^{(u)}(e_n)=\left(\frac{1}{M}+o(1)\right)T
$$
for all $u\in \{1,\ldots,M\}$, we say that $(e_n)$ is (asymptotically) {\em balanced}.
In this section we study the balance of the sequences $(s_n)$ and $(t_n)$. We start with a preliminary result.

\subsection{A preliminary result}

\begin{lemma}\label{Sec2MC}  Let $\mathcal{R}$
be a subset of $\mathbb{Z}_q=\{0,1,\ldots, q-1\}$.
Define
$$
c(i)=\left\{\begin{array}{ll}
+1, & i\in \mathcal{R}, \\
-1, & i\in\mathbb{Z}_{q}\setminus \mathcal{R}.
\end{array}
\right.
$$
For any positive integer $s$ with $1\leq s\leq q$
and any $(\varepsilon_{0},\ldots,\varepsilon_{s-1})\in\{-1,+1\}^s$,
we put
$$
\Gamma(\varepsilon_{0},\ldots,\varepsilon_{s-1})=\left|\left\{n=0,1,\ldots,q-s: \ c(n+i)=\varepsilon_{i}, \
i=0,\ldots, s-1\right\}\right|.
$$
Let $z=z(\varepsilon_{0},\ldots,\varepsilon_{s-1})$ be the number of $i$ with $\varepsilon_i=1$,
$i=0,\ldots,s-1$.
Then we have
\begin{eqnarray*}
\Gamma(\varepsilon_{0},\ldots,\varepsilon_{s-1})=
\left(\frac{|\mathcal{R}|}{q}\right)^{z}\left(1-\frac{|\mathcal{R}|}{q}\right)^{s-z}q
+O\left(2^sC(\mathcal{R},q,s)\right),\quad q\rightarrow \infty.
\end{eqnarray*}
\end{lemma}

\begin{proof} Since otherwise the result is trivial we may assume $\mathcal{R}\not=\emptyset$ and $s\le \log q$.
Let $I\subseteq \{0,\ldots,s-1\}$ be the set of indices of size $z$ satisfying
$$
\varepsilon_{i}=\left\{\begin{array}{cc} 1, & i\in I,\\ -1, & i\not\in I.\end{array}\right.
$$
Note that
$$
f_{\mathcal{R}}(n)+\frac{\mathcal{|R|}}{q}
=\left\{\begin{array}{ll}1, & c(n)=1,\\
0, & c(n)=-1,
\end{array}\right.
$$
where $f_{\mathcal{R}}(n)$ is defined by $(\ref{en})$.
Then
\begin{eqnarray*}
\Gamma(\varepsilon_{0},\ldots,\varepsilon_{s-1})
&=&\sum_{n=0}^{q-s}\prod_{i\in I}\left(f_{\mathcal{R}}(n+i)+\frac{|\mathcal{R}|}{q}\right)
\prod_{i\not\in I}\left(1-\frac{|\mathcal{R}|}{q}-f_{\mathcal{R}}(n+i)\right)\\
&=&\left(\frac{|\mathcal{R}|}{q}\right)^{z}\left(1-\frac{|\mathcal{R}|}{q}\right)^{s-z}(q-s+1)\\
&&+\sum_{
{H\subseteq I \atop J\subseteq
\{0,\ldots,s-1\}\setminus I}\atop H\cup J\not=\emptyset}\left(\frac{|\mathcal{R}|}{q}\right)^{z-|H|}
\left(1-\frac{|\mathcal{R}|}{q}\right)^{s-z-|J|}(-1)^{|J|}\sum_{n=0}^{q-s}\prod_{i\in H\cup J}
f_{\mathcal{R}}(n+i).
\end{eqnarray*}
Since the absolute value of the sums over $n$ can be estimated by $C_{|H\cup J|}(\mathcal{R},q)\le C(\mathcal{R},q,s)$
and the number of pairs~$(H,J)$ to be considered is $2^s-1$ we get the result.
\end{proof}

\subsection{Balance of $(s_n)$}

\begin{theorem}\label{sn1balance}
Let $M\ge 2$ and $T$ be integers with
$$\frac{\log\log\log q}{\log\log q}\leq T\leq q.$$
Let the subset $\mathcal{S}$ of size $T$ and the $M$-ary sequence $(s_n)$ be defined by $(\ref{Sdef})$
and $(\ref{snMdef})$, respectively.
For any integer $u$ with $1\leq u\leq M$ let $N^{(u)}(s_n)$ denote the number of $n=0,1,\ldots,T-2$ with $s_n\equiv u\bmod M$.
Then we have
\begin{eqnarray*}
N^{(u)}(s_n)=\left(\frac{\frac{T}{q}\Big(1-\frac{T}{q}\Big)^{u-1}}{1-\Big(1-\frac{T}{q}\Big)^M}
+o(1)
\right)T+O\left(MC(\mathcal{S},q,M\log\log q)(\log q)^{M-2}\right),\quad q\rightarrow \infty.
\end{eqnarray*}

If $T=o(q)$, then we have
$$
N^{(u)}(s_n)=\left(\frac{1}{M}+o(1)\right)T
+O\left(MC(\mathcal{S},q,M\log\log q)(\log q)^{M-2}\right),\quad q\rightarrow \infty,
$$
and the sequence is $($asymptotically$)$ balanced if
$$C(\mathcal{S},q,M\log\log q)=o\left(\frac{T}{(\log q)^{M-2}}\right).$$
\end{theorem}

\begin{proof} For integers $k\ge 0$ and $u$ with $1\le u\le M$ put
$$
N_{k,u}=\Gamma(1,\underbrace{-1,\ldots,-1}_{Mk},\underbrace{-1,\ldots,-1}_{u-1}, 1).
$$
Choose
$$H=\left\lfloor \log\log q\right\rfloor-4$$
and verify that for  $T\ge \frac{\log\log\log q}{\log\log q} q$
\begin{eqnarray*}
\left(1-\frac{T}{q}\right)^{M(H+1)}&\leq&\left(1-\frac{\log\log\log q}{\log\log q}\right)
^{M\log\log q-4M}\\
&\leq&(\log\log q)^{-M}\left(1-\frac{\log\log\log q}{\log\log q}\right)
^{-4M}=o(1),
\end{eqnarray*}
where we used $1-x\le e^{-x}$, and
$$\frac{\frac{T}{q}\left(1-\frac{T}{q}\right)^{u-1}}{1-\left(1-\frac{T}{q}\right)^{M}}\le 1,\quad 1\le T\le q.
$$
By Lemma \ref{Sec2MC} with $z=2$ and $s=Mk+u+1\le M\log\log q$, $k=0,1,\ldots,H$, we have
\begin{eqnarray*}
N^{(u)}(s_n)&\geq& \sum_{k=0}^{H}N_{k,u}\\
&=&\sum_{k=0}^{H}\left(\left(\frac{T}{q}\right)^{2}\left(1-\frac{T}{q}\right)^{Mk+u-1}q
+O\left(2^{Mk+u+1}C(\mathcal{S},q,Mk+u+1)\right)\right)\\
&=&\left(\frac{\frac{T}{q}\left(1-\frac{T}{q}\right)^{u-1}}{1-\left(1-\frac{T}{q}\right)^{M}}+o(1)\right)T
+O\left(2^{MH+M+1}C(\mathcal{S},q,MH+M+1)\right)\\
&=&\left(\frac{\frac{T}{q}\left(1-\frac{T}{q}\right)^{u-1}}{1-\left(1-\frac{T}{q}\right)^{M}}+o(1)\right)T
+O\left(C(\mathcal{S},q,M\log\log q)(\log q)^{M-2}\right).
\end{eqnarray*}

Note that
\begin{equation}\label{sumu}
\sum_{u=1}^{M}\frac{\frac{T}{q}\left(1-\frac{T}{q}\right)^{u-1}}{1-\left(1-\frac{T}{q}\right)^{M}}
=1.
\end{equation}
Hence, for each $v=1,\ldots,M$ we have
\begin{eqnarray*}
N^{(v)}(s_n)&=&T-1-\sum_{u\ne v} N^{(u)}(s_n)\\
&\le& T -(1+o(1))T
+O\left(MC(\mathcal{S},q,M\log\log q)(\log q)^{M-2}\right)+\frac{\frac{T}{q}\left(1-\frac{T}{q}\right)^{v-1}}{1-\left(1-\frac{T}{q}\right)^M}T
\end{eqnarray*}
and the first result follows.

If $T=o(q)$, then we get
\begin{equation}\label{To}
\frac{\frac{T}{q}\left(1-\frac{T}{q}\right)^{u-1}}{1-\left(1-\frac{T}{q}\right)^M}
=\frac{1+\sum_{j=1}^{u-1}{u-1\choose j}\left(-\frac{T}{q}\right)^j}{M+\sum_{j=2}^M{M\choose j}\left(-\frac{T}{q}\right)^{j-1}}=\frac{1+o(1)}{M+o(1)}=\frac{1}{M}+o(1)
\end{equation}
and the second result follows.
\end{proof}

\subsection{Balance of $(t_n)$}

\begin{theorem}
Let $T$ be an integer with $1\leq T\leq q$.
Let the subset $\mathcal{S}$ and the sequence~$(t_n)$ be defined by $(\ref{Sdef})$
and $(\ref{tndef})$, respectively.
For $v\in\{0,1\}$ let $N^{(v)}(t_n)$ denote the number of $n=0,1,\ldots,T-2$
with $t_n=v$.
Then we have
$$
N^{(v)}(t_n)=\left(1-\frac{T}{q}\right)^{(m-1)(1-v)}\left(1-\left(1-\frac{T}{q}\right)^{m-1}\right)^{v}
T+O\left(2^{m}C(\mathcal{S},q,m)\right),
$$
which shows that the sequence is (asymptotically) balanced if
$$T=\left(1-\frac{1}{2^{1/(m-1)}}\right)q+o(q)\quad\mbox{and}
\quad C(\mathcal{S},q,m)=o(T).$$
\end{theorem}

\begin{proof} By Lemma \ref{Sec2MC} we get
\begin{eqnarray*}
N^{(1)}(t_n)&=&\sum_{u=1}^{m-1}\Gamma(1,\underbrace{-1,\ldots,-1}_{u-1}, 1)\\
&=&\sum_{u=1}^{m-1}\left(\left(\frac{T}{q}\right)^2\left(1-\frac{T}{q}\right)^{u-1}
q+O\left(2^{u+1}C(\mathcal{S},q,u+1)\right)\right) \\
&=&\left(1-\left(1-\frac{T}{q}\right)^{m-1}\right)T+O\left(2^{m}C(\mathcal{S},q,m)\right)
\end{eqnarray*}
and
\begin{eqnarray*}
N^{(0)}(t_n)=T-1-N^{(1)}(t_n)=\left(1-\frac{T}{q}\right)^{m-1}T+O\left(2^{m}C(\mathcal{S},q,m)\right),
\end{eqnarray*}
which completes the proof of the first result.

The second result follows
since
$$\left(1-\frac{T}{q}\right)^{m-1}=\left(1-\left(1-\frac{T}{q}\right)^{m-1}\right)$$
is equivalent to $T=\left(1-\frac{1}{2^{1/(m-1)}}\right)q$.
\end{proof}

\subsection{Pattern distribution of $(s_n)$}

\begin{theorem}\label{sn1pattern}
Let $M\ge 2$ and $T$ be integers with
$$
\frac{\log\log\log q}{\log\log q}q\leq T\leq q.
$$
Let the subset $\mathcal{S}$ of size $T$ and the $M$-ary sequence $(s_n)$ be defined by $(\ref{Sdef})$
and $(\ref{snMdef})$, respectively.
Let $(a_0,\ldots,a_{\ell-1})\in \{1,2,\ldots,M\}^\ell$ be any pattern of fixed length $\ell\geq 1$.
Let $N^{(a_0,\ldots,a_{\ell-1})}(s_n)$ be the number of $n=0,1,\ldots, T-\ell-1$ with
$s_{n+i}=a_i$ for $i=0,\ldots, \ell-1$. Then we have
\begin{eqnarray*}
&&N^{(a_0,\ldots,a_{\ell-1})}(s_n)\\&=&\left(\frac{\Big(\frac{T}{q}\Big)^{\ell}
\Big(1-\frac{T}{q}\Big)^{a_0+\ldots+a_{\ell-1}-\ell}}
{\Big(1-\Big(1-\frac{T}{q}\Big)^{M}\Big)^{\ell}}+o(1)\right)T
+O\left(M^\ell C(\mathcal{S},q,\ell M\log\log q)(\log q)^{\ell M}\right),
\end{eqnarray*}
and the sequence of patterns of length $\ell$ is (asymptotically) balanced
if
$$T=o\left(q\right)\quad \mbox{and}\quad C(\mathcal{S},q,\ell M \log\log q)=o(T/(\log q)^{\ell M}).$$
\end{theorem}

\begin{proof} For non-negative integers $k_0,\ldots, k_{\ell-1}$ we put
\begin{eqnarray*}
&&N'_{k_0,\ldots,k_{\ell-1};a_0,\ldots,a_{\ell-1}}\\
&&=\Gamma(1,\underbrace{-1,\ldots,-1}_{Mk_0},\underbrace{-1,\ldots,-1}_{a_0-1}, 1,
\underbrace{-1,\ldots,-1}_{Mk_1},\underbrace{-1,\ldots,-1}_{a_1-1}, \ldots,1,
\underbrace{-1,\ldots,-1}_{Mk_{\ell-1}},\underbrace{-1,\ldots,-1}_{a_{\ell-1}-1}, 1).
\end{eqnarray*}
Take
$$H=\left\lfloor\log\log q\right\rfloor-2$$
and recall from the proof of Theorem~\ref{sn1balance} that
$$\left(1-\frac{T}{q}\right)^{M(H+1)}
=o(1)\quad\mbox{for}\quad T\ge \frac{\log\log\log q}{\log\log q} q.$$
By Lemma \ref{Sec2MC}
with $s=M(k_0+k_1+\ldots+k_{\ell-1})+a_0+a_1+\ldots+a_{\ell-1}+1\le \ell M\log\log q$ and $z=\ell+1$ we get as in the proof of Theorem~\ref{sn1balance}
\begin{eqnarray*}
&&N^{(a_0,\ldots,a_{\ell-1})}(s_n)\geq \sum_{k_0,k_1,\ldots,k_{\ell-1}=0}^{H}
N'_{k_0,\ldots,k_{\ell-1};a_0,\ldots,a_{\ell-1}}\\
&&=\sum_{k_0,k_1,\ldots,k_{\ell-1}=0}^{H}
\left(\left(\frac{T}{q}\right)^{\ell+1}\left(1-\frac{T}{q}\right)^{M(k_0+\ldots+k_{\ell-1})
+a_0+\ldots+a_{\ell-1}-\ell}q
\right.\\
&&\quad+O\left(2^{M(k_0+\ldots+k_{\ell-1})+a_0+\ldots+a_{\ell-1}+1}
C(\mathcal{S},q,M(k_0+\ldots+k_{\ell-1})+a_0+\ldots+a_{\ell-1}+1)\right)\Bigg)\\
&&=\prod_{i=0}^{\ell-1} \left(\frac{\frac{T}{q}\left(1-\frac{T}{q}\right)^{a_i-1}}
{1-\left(1-\frac{T}{q}\right)^{M}}+o(1)\right)T+O\left(2^{M\ell (H+1)}C(\mathcal{S},q,M\ell (H+1)+1)\right) \\
&&=\left(\frac{\left(\frac{T}{q}\right)^{\ell}\left(1-\frac{T}{q}\right)^{a_0+\ldots+a_{\ell-1}-\ell}}
{\Big(1-\left(1-\frac{T}{q}\right)^{M}\Big)^{\ell}}+o(1)\right)T
+O\left(C(\mathcal{S},q,\ell M\log\log q))(\log q)^{\ell M}\right).
\end{eqnarray*}
From \eqref{sumu} we get
$$
\sum_{a_{0},a_1,\ldots,a_{\ell-1}=1}^{M}
\frac{\left(\frac{T}{q}\right)^{\ell}\left(1-\frac{T}{q}\right)^{a_0+\ldots+a_{\ell-1}-\ell}}
{\left(1-\left(1-\frac{T}{q}\right)^{M}\right)^{\ell}}
=\left(\sum_{u=1}^M \frac{\frac{T}{q}\left(1-\frac{T}{q}\right)^{u-1}}{1-\left(1-\frac{T}{q}\right)^M}\right)^{\ell}
=1
$$
and the first result from
$$N^{(a_0,\ldots,a_{\ell-1})}(s_n)=T-\ell-\sum_{(b_0,\ldots,b_{\ell-1})\in \{1,\ldots,M\}^\ell\setminus\{(a_0,\ldots,a_{\ell-1})\})}
N^{(b_0,\ldots,b_{\ell-1})}(s_n)$$
and the lower bound on $N^{(a_0,\ldots,a_{\ell-1})}(s_n)$.

Finally note that \eqref{To} implies
\begin{eqnarray*}
\frac{\left(\frac{T}{q}\right)^{\ell}\left(1-\frac{T}{q}\right)^{a_0+\ldots+a_{\ell-1}-\ell}}{\left(1-\left(1-\frac{T}{q}\right)^M\right)^\ell}
=\left(\frac{1}{M}+o(1)\right)^\ell
=\frac{1}{M^\ell}+o(1)
\end{eqnarray*}
provided that $T=o\left(q\right)$.
\end{proof}

\subsection{Pattern distribution of $(t_n)$}
\begin{theorem}\label{tnpattern}
 Let $T$ be an integer with
$$
\frac{\log\log\log q}{\log\log q}q\leq T\leq q.
$$
Let the subset $\mathcal{S}$ of size $T$ and the sequence $(t_n)$ be defined by $(\ref{Sdef})$
and $(\ref{tndef})$ respectively.
Let $(b_0,\ldots,b_{\ell-1})\in\{0,1\}^\ell$ be any pattern of length
$\ell\geq 1$.
Let $N^{(b_0,\ldots,b_{\ell-1})}(t_n)$ be the number of $n=0,1,\ldots, T-\ell-1$ such that
$t_{n+i}=b_i$ for $i=0,\ldots, \ell-1$. Then we have
\begin{eqnarray*}
N^{(b_0,\ldots,b_{\ell-1})}(t_n)&=&\left(\Big(1-\frac{T}{q}\Big)^{(m-1)(\ell-b_0-\ldots-b_{\ell-1})}
\Big(1-\Big(1-\frac{T}{q}\Big)^{m-1}\Big)^{b_0+\ldots+b_{\ell-1}}+o(1)\right)T\\
&&+O\left(2^{\ell(m-2)}C(\mathcal{S},q,\ell\log\log q)(\log q)^{\ell}\right)
\end{eqnarray*}
and the sequence of patterns of length $\ell$ is (asymptotically) balanced if
$$T=\left(1-\frac{1}{2^{1/(m-1})}\right)q+o(q)\quad\mbox{and}\quad C(\mathcal{S},q,\ell \log\log q)=o(T/(\log q)^{\ell}).$$
\end{theorem}

\begin{proof}
For positive integers $x_0,\ldots, x_{\ell-1}$
we put
$$
N''_{x_0,\ldots, x_{\ell-1}}
=\Gamma(1,\underbrace{-1,\ldots,-1}_{x_0-1}, 1, \ldots,1,
\underbrace{-1,\ldots,-1}_{x_{\ell-1}-1}, 1).
$$
Let $Z=b_0+\ldots+b_{\ell-1}$ be the number of $j=0,\ldots,\ell-1$ with $b_j=1$.
Without loss of generality we may assume $b_0=\ldots=b_{Z-1}=1$ and $b_Z=\ldots=b_{\ell-1}=0$.
We choose
$$H=\left\lfloor \log\log q\right\rfloor-4$$
and verify that for  $T\ge \frac{\log\log\log q}{\log\log q} q$,
\begin{eqnarray*}
\left(1-\frac{T}{q}\right)^{H}=o(1).
\end{eqnarray*}
By Lemma \ref{Sec2MC}
with $s=x_0+\ldots+x_{\ell-1}+1\le \ell H+1\le \ell \log\log q$ and $z=\ell+1$
we have
\begin{eqnarray*}
&&N^{(b_0,\ldots,b_{\ell-1})}(t_n)\geq\sum_{x_0,\ldots,x_{z-1}=1}^{m-1}
\sum_{x_z,\ldots,x_{\ell-1}=m}^{H}
N''_{x_0,\ldots, x_{\ell-1}}\\
&&=\sum_{x_0,\ldots,x_{Z-1}=1}^{m-1}\sum_{x_Z,\ldots,x_{\ell-1}=m}^{H}
\Big(\Big(\frac{T}{q}\Big)^{\ell}\Big(1-\frac{T}{q}\Big)^{
x_0+\ldots+x_{\ell-1}-\ell}T+O\Big(2^{x_0+\ldots+x_{\ell-1}}
C(\mathcal{S},q, \ell H+1)
\Big)\Big)\\
&&=\left(\left(1-\frac{T}{q}\right)^{m-1}-\left(1-\frac{T}{q}\right)^{H}\right)^{\ell-Z}
\left(1-\left(1-\frac{T}{q}\right)^{m-1}\right)^Z T\\
&&\quad+ O\left(2^{\ell(m+H+1)}C(\mathcal{S},q, \ell \log\log q)\right)\\
&&=\left(\Big(1-\frac{T}{q}\Big)^{(m-1)(\ell-Z)}
\Big(1-\Big(1-\frac{T}{q}\Big)^{m-1}\Big)^{Z}+o(1)\right)T
+O\left(2^{\ell (m-3)}C(\mathcal{S},q,\ell \log\log q)(\log q)^{\ell}\right).
\end{eqnarray*}
Note that
$$
\sum_{b_{0},b_1,\ldots,b_{\ell-1}=0}^{1}
\Big(1-\frac{T}{q}\Big)^{(m-1)(\ell-b_0-\ldots-b_{\ell-1})}
\Big(1-\Big(1-\frac{T}{q}\Big)^{m-1}\Big)^{b_0+\ldots+b_{\ell-1}}
=1.
$$
Thus we get Theorem \ref{tnpattern}.
\end{proof}

\subsection{Pattern distribution of $(u_n)$}
The following result seems to be well-known at least for some special sets such as the set of primitive roots modulo a prime.  However, for the convenience of the reader we add its short proof.

\begin{theorem}\label{unpattern} Let $\mathcal{S}\subseteq \mathbb{Z}_q$ be of size $T$
and the sequence $(u_n)$ be defined by $(\ref{undef})$.
Let $(b_0,\ldots,b_{\ell-1})\in \{0,1\}^\ell$ be a pattern of length $\ell\geq 1$.
Let $N^{(b_0,\ldots,b_{\ell-1})}(u_n)$ be the number of $n=0,1,\ldots, q-\ell$ with
$u_{n+i}=b_i$ for $i=0,\ldots, \ell-1$. Then we have
\begin{eqnarray*}
N^{(b_0,\ldots,b_{\ell-1})}(u_n)=\left(\frac{T}{q}\right)^{b_0+\ldots+b_{\ell-1}}
\left(1-\frac{T}{q}\right)^{\ell-b_0-\ldots-b_{\ell-1}}q+O\left(2^{\ell}C(\mathcal{S},q,\ell)\right)
\end{eqnarray*}
and the sequence of patterns of length $\ell$ is (asymptotically) balanced if
$$T=\frac{q}{2}+o(q)\quad\mbox{and}\quad C(\mathcal{S},q,\ell)=o(q).$$
\end{theorem}

\begin{proof} Put $w=b_0+\ldots+b_{\ell-1}$ and $\varepsilon_i=(-1)^{b_i+1}$ for $i=0,\ldots,\ell-1$.
By Lemma \ref{Sec2MC} we have
\begin{eqnarray*}
N^{(b_0,\ldots,b_{\ell-1})}(u_n)&=&\Gamma(\varepsilon_{0},\ldots,\varepsilon_{\ell-1})\\
&=&\left(\frac{T}{q}\right)^{w}
\left(1-\frac{T}{q}\right)^{\ell-w}q+O\left(2^{\ell}C(\mathcal{S},q,\ell)\right)
\end{eqnarray*}
which implies the result.
\end{proof}

\section{A primer on pseudorandom subsets}\label{sec:examples}

Many pseudorandom subsets have been constructed and studied using number theoretic methods.
We can derive large families of (asymptotically) balanced sequences with uniform pattern distribution
from these pseudorandom subsets using our Theorems~\ref{sn1balance} to \ref{unpattern}.


Dartyge and S\'{a}rk\"{o}zy \cite{DartygeS2007_1} constructed pseudorandom subsets
using $d$th power residues modulo~$p$ for a divisor $d$ of $p-1$.

\begin{proposition}\label{subset-powerresidues} Let $p\geq 2$ be a prime
number, $d\mid p-1$ and let $f\in \mathbb{Z}_p[x]$ be a non-constant polynomial with no multiple roots in
the algebraic closure $\overline{\mathbb{Z}}_p$ of $\mathbb{Z}_p$. Define
$$
\mathcal{R}=\left\{n:\ 0\leq n\leq p-1, \ \exists y\in \mathbb{Z}_p^*
\hbox{\ with \ } f(n)\equiv y^d \bmod  p\right\}.
$$
Then we have
\begin{equation}\label{cardWeil}\left||\mathcal{R}|-\frac{p-m}{d}\right|\le \frac{d-1}{d}(\deg(f)-1)p^{\frac{1}{2}},
\end{equation}
where $m\le \deg(f)$ is the number of zeros of $f(x)$ in $\mathbb{Z}_p$.

Moreover, suppose that at least one
of the following conditions is satisfied:

\textup{(i)} $k\leq 2$;

\textup{(ii)} $d$ is a prime divisor of $p-1$ and $(4k)^{\deg(f)}<p$;

\textup{(iii)} the polynomial $x^{p-1}+\cdots+x+1$ is irreducible in
$\mathbb{Z}_{t}[x]$ for any prime divisor $t$ of $d$ and $\max(\deg(f), k)<p$. \\
Then we have
$$
C_{k}(\mathcal{R},p)\ll (1+o(1))^k \left(1-\frac{1}{d}\right)^{k}\deg(f)kp^{\frac{1}{2}}\log p.
$$
\end{proposition}

Remarks.
\begin{enumerate}
\item  Note that \eqref{cardWeil} is slightly better than \cite[(1.6)]{DartygeS2007_1} and can be obtained from
$$|V_f|-\frac{p-m}{d}=\frac{1}{d}\sum_{j=1}^{d-1}\sum_{n=0}^{p-1} \chi^j(f(n))$$
and the Weil bound for complete multiplicative character sums, where $\chi$ is any character of order $d$ of $\Z_p^*$.
\item In the special case $p>2$, $d=2$ and $f(x)=x$
the set ${\cal R}$ is the set of quadratic residues modulo $p$.
Combining Theorem~\ref{sn1balance} with $M=2$, Theorem~\ref{tnpattern} with $m=2$ and Theorem~\ref{unpattern}, respectively, we recover essentially, that is, up to logarithmic and $o(1)$ terms, the following results:
\begin{enumerate} \item \cite[Theorem~2]{Winterhof2022} on the imbalance of the sequence $(s_n)$,
\item \cite[Theorem~3]{Winterhof2022} on the uniform pattern distribution of $(t_n)$,
\item \cite[Proposition~2]{Ding1998} on the uniform pattern distribution of the Legendre sequence $(u_n)$.
\end{enumerate}
The reason for the slightly weaker results is our generic approach whereas in special cases we can optimize, for example, the choice of $H$ and can deal with complete sums over $\Z_p$ instead of incomplete ones.
\item  For sufficiently large $p$, $f(x)=x$ and, say,
$$\log\log\log p\le d\le \frac{\log\log p}{\log\log\log p}$$
combining Theorem~\ref{sn1pattern} and Proposition~\ref{subset-powerresidues} we see that the sequence $(s_n)$ has (asymptotically) a uniform pattern distribution.

\end{enumerate}

Dartyge, S\'{a}rk\"{o}zy and Szalay \cite{DartygeSS2010} studied the
pseudorandomness of subsets related to primitive roots modulo $p$.

\begin{proposition}\label{subset-primitiveroots}
Let $p$ be an odd prime and
let $\mathcal{G}_p$ be the set of the primitive roots modulo $p$. Let
$s,r\in \mathbb{N}$ with $s\mid p-1$, $r\mid p-1$ and $f(x)\in \mathbb{Z}_p[x]$. Define the subset
$\mathcal{R}\subseteq \mathbb{Z}_p$ by
$$
\mathcal{R}=\left\{g^s:\ g\in \mathcal{G}_p, \exists x\in \mathbb{Z}^*_p
\hbox{\ with \ } f(g^s)=x^r\right\}.
$$
Then we have
$$|\mathcal{R}|=\frac{1}{r}\varphi\left(\frac{p-1}{s}\right)
+O\left(\deg(f)2^{\omega\left(\frac{p-1}{s}\right)}p^{\frac{1}{2}}\log p\right),$$
where $\omega(n)$ denotes the number of distinct prime factors of $n$.

Suppose that $f(x)$ is irreducible over $\mathbb{Z}_p$, $\deg(f)\geq 2$ or $\deg(f)=r=1$.
Then we have
\begin{equation*}
C_k(\mathcal{R},p)\ll (1+o(1))^k
k\deg(f)2^{k\omega\left(\frac{p-1}{s}\right)}
p^{\frac{1}{2}}\log p.
\end{equation*}
\end{proposition}

Remarks.
\begin{enumerate}
    \item We have $2^{\omega(n)}=n^{o(1)}$, see for example \cite{hw}, and thus
    $$2^{k\omega((p-1)/s)}p^{1/2}\log p\le p^{\frac{1}{2}+o(1)}.$$
\item For $r=s=1$ and $f(x)=x$ we get the set of primitive roots modulo $p$. Combining Proposition~\ref{subset-primitiveroots} with
Theorem~\ref{sn1balance}, Theorem~\ref{sn1pattern} and Theorem~\ref{unpattern} with $M=m=2$, respectively, we recover essentially
\begin{enumerate}
    \item \cite[Theorems~1 and 2]{Winterhof2021} on the (im-)balance and pattern distribution of the sequence~$(s_n)$,
    \item \cite[Theorem 1]{Cobeli1998} on the pattern distribution of the sequence $(u_n)$.
\end{enumerate}
\end{enumerate}

Dartyge, S\'{a}rk\"{o}zy and Szalay \cite{DartygeSS2010} also studied the
pseudorandomness of subsets defined by index properties, that is, elements with polynomial values in geometric progression.

\begin{proposition}\label{subsets:index}
Let $p$ be a prime, $f(x)\in \mathbb{Z}_p[x]$ with $\deg(f)\geq 1$,
$r\in \mathbb{Z}$, $s\in \mathbb{N}$, $s<p$. Define
$\mathcal{R}\subseteq \mathbb{Z}_p$ by
$$
\mathcal{R}=\{n:\ 0\leq n\leq p-1, \ \exists h\in\{r,r+1,\cdots,r+s-1\}
\hbox{\ with \ } \textup{ind}\ f(n)\equiv h \bmod p\},
$$
where $\textup{ind}\ m$ denotes the base $g$ index of $m$ to a fixed primitive root $g$
modulo $p$. Then we have
$$|\mathcal{R}|=s+O\left(\deg(f)p^{\frac{1}{2}}\log p\right).$$
Moreover, suppose that at least one
of the following conditions holds:

\textup{(i)} $f$ is irreducible;

\textup{(ii)} if $f$ has the factorization $f=f_1^{\alpha_1}\ldots f_u^{\alpha_u}$
where $\alpha_i\in\mathbb{N}$ and $f_i$ is irreducible over $\mathbb{Z}_p$, then there
exists a $\beta$ such that exactly one or two $f_i$'s have the degree $\beta$;

\textup{(iii)} $k\leq 2$;

\textup{(iv)} $(4k)^{\deg(f)}<p$ or $(4\deg(f))^{k}<p$. \\
Then we have
$$
C_{k}(\mathcal{R},p)\ll (1+o(1))^k \deg(f)k2^{k}p^{\frac{1}{2}}(\log p)^{k+1}.
$$
\end{proposition}

Remarks.
\begin{enumerate}\item This construction extends Gyarmati's earlier construction \cite{Gyarmati2004} for $r=1$ and $s=(p-1)/2$.
\item Propositions~\ref{subset-powerresidues} and~\ref{subset-primitiveroots} provide only non-trivial constructions of size at most $(p-1)/2$. Hence, Theorem~\ref{tnpattern} can guarantee uniform pattern distribution only for $m=2$.
However, the size $s$ in the construction of Proposition ~\ref{subsets:index} is very flexible and taking
$$s=\left\lfloor \left(1-\frac{1}{2^{1/(m-1)}}\right)p\right\rfloor$$
we get a uniform pattern distribution for any $m\ge 2$.
\end{enumerate}

Dartyge, Mosaki and
S\'{a}rk\"{o}zy \cite{DartygeMS2009} presented
 the analogs for arithmetic progressions and inverses of arithmetic progressions.

\begin{proposition}
Assume that $p$ is an odd
prime number, $f(x)\in \mathbb{Z}_p[x]$ is of degree $\deg(f)\geq 2$. Let
$r\in \mathbb{Z}$, $s\in \mathbb{N}$, $s<p$. Define
$\mathcal{R}\subseteq \mathbb{Z}_p$ by
$$
\mathcal{R}=\{n:\ 0\leq n\leq p-1, \ \exists h\in\{r,r+1,\cdots,r+s-1\}
\hbox{\ with \ }f(n)\equiv h \bmod p\}.
$$
Then we have
$$|\mathcal{R}|=s+O\left(\deg(f)p^{\frac{1}{2}}\log p\right)$$
and for $k\leq \deg(f)-1$
$$
C_{k}(\mathcal{R},p)\ll \deg(f)p^{\frac{1}{2}}(\log p)^{k+1}.
$$
\end{proposition}

\begin{proposition}
Assume that $p$ is an odd
prime number, $r\in \mathbb{Z}$, $s\in
\mathbb{N}$, $s<p$, $f(x)\in \mathbb{Z}_p[x]$ has no multiple root
and $1\leq \deg(f(x))<p$. Define $\mathcal{R}\subseteq \mathbb{Z}_p$ by
$$
\mathcal{R}=\{n:\ 0\leq n\leq p-1, \gcd(f(n),p)=1, \exists
h\in\{r,r+1,\cdots,r+s-1\} \hbox{\ with \ } hf(n)\equiv 1 \bmod
p\}.
$$
Then we have
$$|\mathcal{R}|=s+O\left(\deg(f)p^{\frac{1}{2}}\log p\right)$$
and for $k<\frac{p}{2\deg(f)}$
$$
C_{k}(\mathcal{R},p)\ll \deg(f)p^{\frac{1}{2}}(\log p)^{k+1}.
$$
\end{proposition}

Dartyge and S\'{a}rk\"{o}zy \cite{DartygeS2009} presented constructions by using
the argument of complex numbers, multiplicative and
additive characters.

\begin{proposition}
Let $p$ be a prime,
$\chi$ a multiplicative character, $\psi$ a non-trivial additive character
modulo $p$, $f(x),g(x)\in\mathbb{Z}_p[x]$ with $\deg(g)\geq 2$, and $\alpha,\beta$ real numbers with
$\alpha<\beta\leq \alpha+1$. Define $\mathcal{R}\subseteq \mathbb{Z}_p$ by
$$
\mathcal{R}=\{n:\ 0\leq n\leq p-1, \gcd(f(n),p)=1  \hbox{\ and \ }
2\pi \alpha\leq\arg\left(\chi(f(n))\psi(g(n))\right)<2\pi \beta\}.
$$
Then we have
$$|\mathcal{R}|=(\beta-\alpha)p+O\left((\deg(f)+\deg(g))p^{\frac{1}{2}}\log p\right)$$
and for $k\leq \deg(g)-1$
$$
C_{k}(\mathcal{R},p)\ll (1+o(1))^k(\deg(f)+\deg(g))p^{\frac{1}{2}}(\log p)^{k+1}.
$$
\end{proposition}

Remarks.
\begin{enumerate}\item This construction generalizes Dartyge's and S\'{a}rk\"{o}zy's
construction \cite{DartygeS2007_1} for $\chi$ being the Legendre symbol, $\psi$ trivial,
$\alpha=-1/4$ and $\beta=1/4$, and extends Dartyge's, Mosaki's and
S\'{a}rk\"{o}zy's construction \cite{DartygeMS2009} for $\chi$ trivial, $\psi$ canonical,
$\alpha=\frac{r}{p}$ and $\beta=\frac{r+s-1}{p}$.
\end{enumerate}

Let $p$ be a prime and let $n$ be an integer
with $\gcd(n, p)=1$. The Fermat quotient~$q_p(n)$ is defined as the
unique integer with
$$
q_p(n)\equiv \frac{n^{p-1}-1}{p} \bmod p,\quad 0\leq q_p(n)\leq
p-1.
$$
We also define $q_p(n)=0$ for $\gcd(n,p)>1$. The first author and Zhang \cite{LiuZ2019}
studied the pseudorandomness of subsets constructed by Fermat
quotients using estimates for exponential sums and character
sums with Fermat quotients.

\begin{proposition}
Let $p$ be an odd prime
number, and $d\mid p-1$. Define $\mathcal{R}\subseteq
\mathbb{Z}_{p^2}$ by
$$
\mathcal{R}=\left\{n:\ 0\leq n\leq p^2-1, \ \exists \ y \hbox{\ such
that \ } 1\leq y\leq p-1 \hbox{\ and \ } q_p(n)\equiv y^d \bmod
p\right\}.
$$
Then we have
$$|\mathcal{R}|=\frac{(p-1)^2}{d}$$
and
$$
C_{k}\left(\mathcal{R},p^2\right)\ll kp^{\frac{5}{3}}.
$$
\end{proposition}

\begin{proposition}\label{subsets:Fermatquotient2}
Let $p$ be an odd prime
number, and let $\mathcal{G}_p$ be the set of the primitive roots
modulo $p$. Define $\mathcal{R}\subset
\mathbb{Z}_{p^2}$
by
$$
\mathcal{R}=\left\{n:\ 0\leq n\leq p^2-1, \ q_p(n)\in
\mathcal{G}_p\right\}.
$$
Then we have
$$|\mathcal{R}|=(p-1)\varphi(p-1)$$ and
$$
C_{k}\left(\mathcal{R},p^2\right)\ll k
2^{k\omega(p-1)}p^{\frac{5}{3}}.
$$
\end{proposition}

Now it is easy to plug the results of Propositions~\ref{subset-powerresidues} to \ref{subsets:Fermatquotient2} in Theorems~\ref{sn1balance} to \ref{unpattern} to get many new results on balance and uniform pattern distribution of sequences of type $(s_n)$, $(t_n)$ and $(u_n)$.

\section*{\bf \Large Acknowledgments}

The first author is supported by National Natural Science Foundation of China under Grant No. 12071368,
and the Science and Technology Program of Shaanxi Province of China
under Grant No. 2019JM-573 and 2020JM-026. The second author
was partially supported by the Austrian Science Fund FWF
Project P 30405-N32.

\bigskip\bigskip

\baselineskip=0.9\normalbaselineskip 

\end{document}